\documentclass[11pt,twoside,reqno,a4paper]{amsproc}
\usepackage[utf8]{inputenc}
\usepackage{courier}
\usepackage[T1]{fontenc}

\usepackage[scaled]{helvet}
\usepackage[mathbf]{euler}
\usepackage{a4wide}
\usepackage[style=alphabetic,backend=bibtex,maxnames=5,maxalphanames=5,doi=false,isbn=false,url=false]{biblatex}
\bibliography{sources.bib}

\renewbibmacro{in:}{}
\renewbibmacro*{issue+date}{%
 \ifboolexpr{not test {\iffieldundef{year}} or not test {\iffieldundef{issue}}}
    {\printtext[parens]{%
       \iffieldundef{issue}
         {\usebibmacro{date}}
          {\printfield{issue}%
          \setunit*{\addspace}%
          \usebibmacro{date}}}}
    {}
 \newunit}
\usepackage{hyperref}
\usepackage{enumerate}

\usepackage{amssymb}
\usepackage{amsopn}
\usepackage{amsmath}
\usepackage{mathtools}
\usepackage{tikz}
\usepackage{tikz-cd}
\usepackage{ifthen}
\usepackage{graphics}
\usepackage{amsthm}
\usepackage{topthm}

\newtheorem{bigthm}{Theorem}
\newtheorem{bigcor}[bigthm]{Corollary}
\newtheorem{thm}{Theorem}[section]
\newtheorem{lem}[thm]{Lemma}
\newtheorem{prop}[thm]{Proposition}
\newtheorem{cor}[thm]{Corollary}
\newtheorem{question}[thm]{Question}

\theoremstyle{definition}
\newtheorem{dfn}[thm]{Definition}

\theoremstyle{remark}
\newtheorem{ex}[thm]{Example}

\newtheorem{rem}[thm]{Remark}

\usepackage{todonotes}

\author{Thorben Kastenholz and Jens Reinhold}
\address{Mathematisches Institut, Bunsenstra{\ss}e 3-5, 37073 G{\"o}ttingen, Germany}
\email{\texttt{thorben.kastenholz@mathematik.uni-goettingen.de}}
\address{Mathematisches Institut, Einsteinstr. 62, 48149 M{\"u}nster, Germany}
\email{\texttt{jens.reinhold@posteo.de}}

\thanks{}

\title{Simplicial volume and essentiality of manifolds fibered over spheres
\blfootnote{
 2020 Math subject classification: primary: 53C23 secondary: 57N65, 57M07, 20J05\\
 Keywords: bounded cohomology and simplicial volume, group homology, mapping tori, fiber bundles, positive scalar curvature, macroscopic dimension, characteristic classes
}}
\newcommand{\introduce}[1]
  {\emph{#1}}

\newcommand\blfootnote[1]{%
  \begingroup
  \renewcommand\thefootnote{}\footnote{#1}%
  \addtocounter{footnote}{-1}%
  \endgroup
}

\newcommand{\apply}[2]
  {{#1}\!\left({#2}\right)}
\newcommand{\at}[2]
  {\left.{#1}\right\rvert_{#2}}
\newcommand{\Identity}%
  {\mathrm{Id}}
\newcommand{\NaturalNumbers}%
  {\mathbf{N}}
\newcommand{\Integers}%
  {\mathbf{Z}}
\newcommand{\Rationals}%
  {\mathbf{Q}}
\newcommand{\Reals}%
  {\mathbf{R}}
  \newcommand{\ComplexNumbers}%
  {\mathbf{C}}
\newcommand{\AbstractProjection}[1] 
  {p_{#1}}
\newcommand{\RealPart}[1]
  {\apply{\operatorname{Re}}{#1}}
\newcommand{\ImaginaryPart}[1]
  {\apply{\operatorname{Im}}{#1}}

\newcommand{\Manifold}%
  {M}
  \newcommand{\ManifoldAlternative}%
  {N}
 \newcommand{\ManifoldAuxiliary}%
  {K}
\newcommand{\NullBordism}%
  {W}
\newcommand{\Bordism}
  {P}
 \newcommand{\ManifoldFiber}%
 {M}
 \newcommand{\ManifoldTotal}%
 {E}
  \newcommand{\ManifoldBase}%
  {B}
\newcommand{\SmoothMap}%
  {\phi}
  \newcommand{\MorseFunction}%
  {f}
\newcommand{\Diffeomorphism}%
  {\Phi}
\newcommand{\Dimension}%
  {d}
  \newcommand{\HalfDimension}%
  {n}
\newcommand{\FundamentalClass}[1]
  {\left[#1\right]}
\newcommand{\Interval}%
  {I}
\newcommand{\Ball}[1]
  {D^{#1}}
\newcommand{\Sphere}[1] 
  {S^{#1}}
\newcommand{\Torus}[1]
  {T^{#1}}
\newcommand{\SimplicialVolume}[1]
  {\lvert \lvert #1 \rvert \rvert}
\newcommand{\ellone}%
  {\ell_{1}}

\newcommand{\SeparatingManifold}%
  {S}

\newcommand{\MacroscopicDimension}[1]
  {\text{mc}({#1})}
\newcommand{\mc}[1]
  {\text{mc}(\widetilde{#1})}

\newcommand{\Diff}[1]
  {\mathrm{Diff}\!\left(#1\right)}
  \newcommand{\DiffGroup}[1]
  {\mathrm{Diff}^{B\Group}\!\left(#1\right)}
\newcommand{\DiffZero}[1]
  {\mathrm{Diff}_0\!\left(#1\right)}
\newcommand{\DiffOne}[1]
  {\widetilde{\mathrm{Diff}}_0\!\left(#1\right)}

\newcommand{\FiberingSpace}%
  {E}
\newcommand{\FiberingProjektion}[1] 
  {\pi_{#1}}
\newcommand{\Fiber}%
  {F}
\newcommand{\FiberDimension}%
  {m}
\newcommand{\Base}%
  {B}
\newcommand{\ClutchingFunction}[1] 
  {\varphi_{#1}}

\newcommand{\Group}%
  {\Gamma}
\newcommand{\GroupElement}%
  {g}

\newcommand{\HomologyClass}%
  {\alpha}
\newcommand{\HomologyOfSpaceObject}[3]
  {\apply{H_{#1}}{#2 ; #3}}
  \newcommand{\CohomologyOfSpaceObject}[3]
  {\apply{H^{#1}}{#2 ; #3}}
  \newcommand{\BoundedCohomologyOfSpaceObject}[3]
  {\apply{H^{#1}_{\text{b}}}{#2 ; #3}}
\newcommand{\HomologyOfSpaceMorphism}[1]
  {{#1}_{\ast}}
\newcommand{\HomologyOfGroupObject}[3]
  {\apply{H_{#1}}{#2; #3}}
\newcommand{\HomologyOfGroupMorphism}[2]
  {{#1}_{\ast}}
\newcommand{\HomologyOfSpacePairObject}[3]
  {\apply{H_{#1}}{{#2},{#3}}}
\newcommand{\Multiple}%
  {\lambda}

\newcommand{\TopologicalSpace}%
  {X}
\newcommand{\Point}%
  {\ast}
\newcommand{\Loop}%
  {\gamma}
\newcommand{\ContinuousMap}%
  {f}
  \newcommand{\ContinuousMapALT}%
  {g}
 \newcommand{\maps}%
  {\ensuremath{\text{maps}}}
\newcommand{\HomotopyGroupOfObject}[3]
  {\apply{\pi_{#1}}{{#2},{#3}}}
\newcommand{\HomotopyGroupOfPairObject}[4]
  {\apply{\pi_{#1}}{{#2},{#3},{#4}}}
\newcommand{\HomotopyGroupMorphism}[1] 
  {{#1}_{\ast}}
\newcommand{\EMSpace}[2]
  {\apply{K}{{#1},{#2}}}
\newcommand{\ClassifyingSpace}[1] 
  {B#1}
\newcommand{\MappingTorus}[1] 
  {T_{#1}}
\newcommand{\UniversalCovering}[1] 
  {\widetilde{#1}}
\newcommand{\UniversalCoveringMap}[1] 
  {\widetilde{#1}}
\newcommand{\SpacePSC}[1]
  {\mathcal R^{\text{psc}}({#1})}

\begin{document}

\begin{abstract}
We study the question when a manifold that fibers over a sphere can be rationally essential, or even have positive simplicial volume. More concretely, we show that
mapping tori of manifolds (whose fundamental groups can be quite arbitrary) of
dimension $2n +1 \geq 7$ with non-zero simplicial volume are very common. This contrasts the case of fiber bundles over a sphere of dimension $\Dimension \ge 2$: we prove that their total spaces are rationally inessential if $\Dimension \geq 3$, and
always have simplicial volume $0$. Using a result by Dranishnikov, we also deduce a surprising property of macroscopic dimension, and we give two applications to positive scalar curvature and characteristic classes, respectively.

\end{abstract}
\maketitle

\section{Introduction}
For a topological space $X$, M.~Gromov defined in \cite{Gromov1}  the $\ellone$
norm of a singular homology class
$
  \HomologyClass
  \in
  \HomologyOfSpaceObject{d}{X}{\Reals}
$
as
\[
  \SimplicialVolume{\HomologyClass}
  \coloneqq
  \inf
  \left\{
    \sum_{j \in J} |\lambda_j|\
    \middle|
    \HomologyClass
    \text{ is represented by the singular cycle }
    \sum_{j \in J} \lambda_j \sigma_j
  \right\}.
\]
The simplicial volume of a closed connected oriented smooth manifold $\Manifold$ is
then defined as
$
  \SimplicialVolume{\Manifold}
  \coloneqq
  \SimplicialVolume{\FundamentalClass{\Manifold}}
$,
where $[\Manifold]$ denotes the fundamental class of $\Manifold$. All manifolds
in this paper are assumed to have the properties just mentioned, unless explicitly
stated otherwise.

Recall the notion of smooth manifold bundles, meaning locally trivial
fibrations whose transition maps are diffeomorphisms. A classical result by
Gromov \cite{Gromov1} implies that for any oriented manifold bundle
$\ManifoldFiber \to \ManifoldTotal \to \ManifoldBase$ such that
$\HomotopyGroupOfObject{1}{\ManifoldFiber}{\Point}$ is amenable (and $\Manifold$ has positive dimension),
$\SimplicialVolume{\ManifoldTotal}$ vanishes. However, no similar
condition concerning the base $\ManifoldBase$ ensures that the simplicial
volume
$\SimplicialVolume{\ManifoldTotal}$ vanishes; for instance, a mapping torus
of a pseudo-Anosov surface automorphism is a hyperbolic $3$-manifold
\cite{Thurston, Otal}, and thus has
non-zero simplicial volume by a result by Gromov (again see \cite{Gromov1}).
This observation leads to the following concise question that inspired this
work.
\begin{question}
\label{que:1}
  For positive integers $\Dimension, \FiberDimension$ and an oriented manifold bundle
  $
    \ManifoldFiber
    \to
    \ManifoldTotal
    \to
    \Sphere{\Dimension}
  $ whose fiber $\ManifoldFiber$ is $\FiberDimension$-dimensional, which
  values can the simplicial volume of $\ManifoldTotal$ attain?
\end{question}

The main purpose of this note is to address Question \ref{que:1} by proving the following:

\begin{itemize}
\item
  For $d=1$ and $m \geq 6$ even, the set of values is dense. (Corollary
  \ref{cor:mappingtori})
\item
  For $d > 1$, the simplicial volume of the total space always vanishes.
  (Theorem \ref{thm:VanishingSimplicialVolume})
\end{itemize}

An vital step along the way to prove Corollary \ref{cor:mappingtori} is the following result which could be of independent interest and perhaps also useful beyond the applications in this note.

\begin{bigthm}[Mapping tori representing group homology classes]
\label{bigthm:A}
  Let $\Group$ be a finitely presented group and let
  $
    \HomologyClass
    \in
    \HomologyOfSpaceObject{2n+1}{\Group}{\Rationals}
  $,
  where $2n+1 \geq 7$, denote a non-zero homology class.
  Then there exists a closed, connected, oriented $2n$-manifold
  $\SeparatingManifold$ with
  $\HomotopyGroupOfObject{1}{\SeparatingManifold}{\Point} \cong \Group$
  and a diffeomorphism $\Diffeomorphism$ of $\SeparatingManifold$ that
  fixes a point $\Point \in \SeparatingManifold$ and acts trivially on
  $\HomotopyGroupOfObject{1}{\SeparatingManifold}{\Point}$, such that
  the resulting mapping torus
  $\MappingTorus{\Diffeomorphism}$ admits a map $\ContinuousMap\colon
  \MappingTorus{\Diffeomorphism} \to \ClassifyingSpace{\Group}$ which induces
  the projection to $\Group\times \Integers \to \Group$ on fundamental groups
  and satisfies
  $
    \apply
      {\HomologyOfSpaceMorphism{\ContinuousMap}}
      {\FundamentalClass{\MappingTorus{\Diffeomorphism}}}
    =
    \lambda
    \HomologyClass
  $
  for some non-zero integer $\lambda$.
\end{bigthm}

The main ingredient in the proof of Theorem \ref{bigthm:A} is a classical but perhaps a little overlooked observation
by Lawson: every odd-dimensional manifold can be represented as a twisted
double. From Theorem \ref{bigthm:A}, we easily deduce the following answer to Question \ref{que:1}:
\begin{bigcor}[Values of simplicial volumes of mapping tori]
\label{cor:mappingtori}
  Let  $2n+1 \geq
  7$ and let $\nu \in \Reals_{\geq 0}$ be such that there exists a manifold $\Manifold$ of dimension $2n+1$ such that $\SimplicialVolume{\Manifold} =
  \nu$. Then there exists a mapping torus $T$ of some $2n$-manifold such that
  $\SimplicialVolume{T}= \nu$. Moreover, the possible values for the simplicial volume of mapping tori of a fixed dimension $2n+1 \geq 7$ lie densely in $\Reals_{\geq 0}$.
\end{bigcor}
Here the denseness result follows from work by Heuer and Löh in
\cite{LoehHeuer}.

For $\Dimension > 1$ we have the following vanishing result.
\begin{bigthm}[Vanishing simplicial volume of total spaces]
\label{thm:VanishingSimplicialVolume}
  Let
  $
    \ManifoldFiber
    \to
    \ManifoldTotal
    \to
    \Sphere{\Dimension}
  $
  be a smooth fiber bundle with $\Dimension \ge 2$. Then the
  simplicial volume $\SimplicialVolume{\ManifoldTotal}$ is
  zero.
\end{bigthm}

\subsection*{Essentiality}
A notion--also introduced by Gromov and linked to manifold ideas in geometry--that is related to simplicial volume is the following:
\begin{dfn}
  A $\Dimension$-manifold $\ManifoldAlternative$ is called
  \introduce{essential}
  if the classifying map   of its universal covering, 
  $
    \ContinuousMap
    \colon
    \ManifoldAlternative
    \to
    \ClassifyingSpace{\HomotopyGroupOfObject{1}{\ManifoldAlternative}{\Point}}
  $,
satisfies
  \[
    0
    \neq
    \apply
      {\HomologyOfSpaceMorphism{\ContinuousMap}}
      {\FundamentalClass{\ManifoldAlternative}}
    \in
    \HomologyOfSpaceObject
      {d}
      {
        \ClassifyingSpace
        {\HomotopyGroupOfObject{1}{\ManifoldAlternative}{\Point}}
      }
      {\Integers}\!,
  \]%
  otherwise it is called \introduce{inessential}. If the analogous statement with $\Rationals$ coefficients holds, $\ManifoldAlternative$ is called \introduce{rationally essential}; otherwise it is called \introduce{rationally inessential}.

\end{dfn}
Since the classifying map of the universal covering induces an isomorphism on bounded cohomology \cite{Gromov1, Ivanov} and thus--equivalently--on $\ell^1$ homology, the simplicial volume of a manifold $\ManifoldAlternative$ only depends on the homology class $ {\HomologyOfSpaceMorphism{\ContinuousMap}}
      {\FundamentalClass{\ManifoldAlternative}}
    \in
    \HomologyOfSpaceObject
      {d}
      {
        \ClassifyingSpace
        {\HomotopyGroupOfObject{1}{\ManifoldAlternative}{\Point}}
      }
      {\Rationals}$. Thus it becomes evident that rational essentiality is a necessary
condition for the simplicial volume to be positive. This leads to the
following related question, a generalization of which we answer up to
the Novikov conjecture and with the exception of base $\Sphere{2}$.

\begin{question}
\label{que:2}
  Is every manifold $\ManifoldTotal$ which fibers over a simply
  connected sphere rationally inessential?
\end{question}


Recall that a manifold $\ManifoldAlternative$ is called flexible if there
exists a map $\ManifoldAlternative \to \ManifoldAlternative$ of degree
$d\not\in
\{0,\pm1\}$.
\begin{bigthm}[Non-essentiality of total spaces]
\label{thm:NonEssentiality}
  Let $\ManifoldFiber$ be a manifold whose fundamental group
  $\HomotopyGroupOfObject{1}{\ManifoldFiber}{\Point}$
  satisfies the Novikov conjecture, let $\ManifoldBase$ be
  a flexible $2$-connected manifold, and
  let $\ManifoldFiber \to \ManifoldTotal \to \ManifoldBase$
  be an oriented manifold bundle.
  Then $\ManifoldTotal$ is rationally inessential. 
  \end{bigthm}



The purpose of the later sections (3 - 5) is to present three independent applications:
\subsection*{Application to positive scalar curvature.}
Let $\SpacePSC{\Manifold}$ denote the space of Riemannian metrics on
$\Manifold$ whose scalar curvature is strictly positive at every point
('psc'). Questions about the existence and uniqueness of such metrics mean in
practice to decide whether this space is non-empty or contractible,
respectively. Numerous papers cover these topics, for instance see
\cite{HankeSchickSteimle, walsh_hspaces, BERW, crowleyschicksteimle,
erw_psc2}. Classically, one requires the occuring manifolds to be spin in
order to employ index theory to prove such results. Due to these
requirements, it seems as if no example of a manifold $\Manifold$ whose
universal cover has no spin structure was previously known to support
infinitely many non-isotopic psc metrics.
In Section \ref{Sec:PSC} we provide such an example that builds on the ideas we developed for the proof of Theorem \ref{bigthm:A}:

\begin{bigthm}[A non-spin manifold with infinitely many non-isotopic psc metrics]
\label{thm:PSC}
  Let $\MorseFunction$ be the Morse function on $T^7 = (\Sphere{1})^7 \subset
  \ComplexNumbers^7$ that is defined as the sum of real values of the complex
  coordinates and let $\SeparatingManifold = \MorseFunction^{-1}(0)$. Then
  $\pi_0
  (\SpacePSC{\SeparatingManifold \# (\ComplexNumbers P^2 \times \Sphere{2})})$
 is infinite.
\end{bigthm}

\subsection*{Application to macroscopic dimension.} In Section~\ref{Sec:Essentiality} we
combine our results on the simplicial volume of mapping tori with a result by Dranishnikov to conclude a
surprising fact about macroscopic dimension. This notion is 
defined for any metric space $X$ as the smallest nonnegative integer $k$ such
that there exists a Lipschitz map $f \colon X \to Y$ so that $Y$ is a
$k$-dimensional simplicial complex and there is a uniform bound on the
diameters of pre-images $f^{-1}(y), y \in Y$. We use Theorem \ref{bigthm:A}
to prove that in any dimension $2n+1\ge 7$, there exist two manifolds whose
universal coverings are quasi-isometric, but have different macroscopic
dimensions. For a precise formulation of our result, see Theorem
\ref{thm:MacroscopicDimensionAmenable}.

\subsection*{Application to characteristic classes.}
In the final section \ref{Sec:CharacteristicClasses}, we use Theorem
\ref{thm:NonEssentiality} to study certain characteristic classes, akin to
generalized Miller--Morita--Mumford classes, but pertaining to the cohomology
of the fundamental group of the fiber. In recent years, the study of
generalized Milller--Morita--Mumford classes has seen spectacular advances, in
particular the seminal work by Galatius--Randal-Williams that culminated in
\cite{GRW}, which showed that up to stabilization in many cases the cohomology
rings of diffeomorphism groups are generated by these cohomology
classes. Our result takes the fundamental group into account
and has some similarity with the vanishing result proven in
\cite{TautCl-4Authors}.

Suppose $M$ is an oriented manifold of dimension $d$ such that $\Group \coloneqq \pi_1(\ManifoldFiber)$ is non-trivial. The space $\maps(M,B\Group)$ of all maps $\Manifold \to B\Group$ admits a right action by the group $\Diff{\Manifold}$ of orientation-preserving diffeomorphisms of $\Manifold$, and the homotopy quotient
$B\DiffGroup{\Manifold} \coloneqq \maps(\Manifold,B\Group) /\!/ \Diff{\Manifold} \coloneqq \maps(\Manifold,B\Group) \times_{ \Diff{\Manifold}} E \Diff{\Manifold}$
is the base of a smooth $\Manifold$-bundle
\[\pi^{B\Group}\colon \Manifold \to \maps(\Manifold,B\Group) \times_{ \Diff{\Manifold} } \Manifold \to B\DiffGroup{\Manifold}\]
whose total space maps to $B\Group$ via evaluation. Now for
any class $\alpha \in \CohomologyOfSpaceObject{\ast}{B\Gamma}{\Rationals}$ we may form the ``$\kappa$-class'' $\kappa_{\alpha} \coloneqq  \int_{\pi^{B\Group}} \text{ev}^{\ast} \alpha \in \CohomologyOfSpaceObject{\ast-d}{B\DiffGroup{\Manifold}}{\Rationals}.$


\begin{bigthm}[Asphericity of $\kappa$-like classes from the fundamental group]
\label{thm:CharClassesHurewicz}
Let $\Dimension \geq 6$, let $\Gamma$ be a finitely presented group and let $\alpha \in \CohomologyOfSpaceObject{\Dimension + j}{B\Gamma}{\Rationals}$ be non-zero, where $j \geq 0$.
\begin{enumerate}
\item If $\Dimension = 2\HalfDimension$ and $\Gamma$ has the finiteness property $F_{\HalfDimension}$, there exists a closed connected smooth $\Dimension$-manifold $\ManifoldFiber$ such that $\pi_1{\ManifoldFiber} \cong \Gamma$ and $\kappa_{\alpha} \in \CohomologyOfSpaceObject{j}{ B\DiffGroup{\Manifold}}{\Rationals}$ is non-zero.
\item Assume $j \geq 3$. Then for any closed connected smooth $\Dimension$-manifold $\ManifoldFiber$ with $\pi_1{\ManifoldFiber} \cong \Gamma$, $\kappa_{\alpha}$ evaluates trivially on the image of the Hurewicz $\pi_j \left( B\DiffGroup{\Manifold}\right) \to \HomologyOfSpaceObject{j}{ B\DiffGroup{\Manifold}}{\Rationals}$.
\end{enumerate}
\end{bigthm}

\subsection*{Acknowledgments}
We are indebted to S{\o}ren Galatius who showed us the way when this project was in its early phase by suggesting the main idea of Theorem \ref{thm:TwistedDouble}. We are also grateful to Bernhard Hanke for an extensive email correspondence, and to Georg Frenck for an inspiring conversation. Finally, we would like to thank Clara L\"oh, Wolfgang L\"uck, Oscar Randal-Williams, Thomas Schick, and Rudi Zeidler for helpful remarks.

\section{Twisted doubles, mapping tori, and group homology}\label{Sec:MappingTori}
In this section we prove Theorem~\ref{bigthm:A}. In order to represent any (rational)
odd-dimensional group homology class by a mapping torus, we first sketch the proof of a
general structure result for odd-dimensional manifolds, due to Lawson
\cite{Lawson}. To state it, we need the following definition:
\begin{dfn}
  Given an oriented, closed manifold $\SeparatingManifold$ together with an oriented 
  nullbordism $\NullBordism$ of $\SeparatingManifold$ and an orientation
  preserving diffeomorphism
  $
    \Diffeomorphism
    \colon
    \SeparatingManifold
    \to
    \SeparatingManifold
  $, 
  we call $\NullBordism \cup_{\Diffeomorphism} (- \NullBordism)$ the
  \introduce{$\Diffeomorphism$-twisted double of $\NullBordism$}. Here $- \NullBordism$ stands for $W$ with the opposite orientation.
\end{dfn}
Note that the requirement that $\Diffeomorphism$ is orientation-preserving ensures that the twisted double is again oriented.
Lawson proved that any manifold of odd dimension at least $7$
can be represented as a twisted double. Moreover the proof in \cite{Lawson}
yields an even stronger result given below, which was not stated in his work. Since
we need this strengthening, we will provide a sketch of Lawsons proof to justify it:
\begin{thm}[Twisted Double Theorem]
  \label{thm:TwistedDouble}
  Given a closed, orientable manifold $\Manifold$ of dimension $2m+1 \geq 7$,
  there exists a closed orientable manifold $\SeparatingManifold$, a
  pointed orientation preserving diffeomorphism
  $
    \Diffeomorphism
    \colon
    \SeparatingManifold
    \to
    \SeparatingManifold
  $,
  and an oriented nullbordism $\NullBordism$ of $\SeparatingManifold$ such that:
  \begin{enumerate}[(i)]
  \item
    \label{itm:TwistedDoubleSIdentity}
    $\Diffeomorphism$ induces the identity on the fundamental group of
    $\SeparatingManifold$.
  \item
    The $\Diffeomorphism$-twisted double of $\NullBordism$ is diffeomorphic to
    $\Manifold$.
  \item
    \label{itm:TwistedDoubleInclusionsIdentity}
    The inclusion $\SeparatingManifold \hookrightarrow \NullBordism$
    induces an isomorphisms on fundamental groups.
  \end{enumerate}
\end{thm}
\begin{proof}
  We start with a minimal handlebody decomposition of $\Manifold$: we
  denote the union of all handles of index less or equal
  to $m$ by $\NullBordism_1$ and the union of all handles of index bigger or
  equal than
  $m+1$ by $\NullBordism_2$.
  These two submanifolds with boundary intersect in their common boundary
  $\SeparatingManifold$.
  The goal is to find an $m$-dimensional subcomplex
  $K\subset \SeparatingManifold$ such that $K \to \SeparatingManifold \to
  \NullBordism_i$
  is a simple homotopy equivalence.
  For a general handlebody decomposition this is not necessarily possible, but
  it always becomes possible if one adds a sufficient number of canceling $m$
  and $m+1$-handles (or in other words "replaces $\Manifold$ by $\Manifold \#
  \Sphere{2m+1}$"), which is proven in \cite{Lawson}. Let us assume that the
  handlebody decomposition has been stabilized like that.

  Let $V$ denote a regular neighborhood of $K \subset \SeparatingManifold$.
  Since the codimension of $K\subset \SeparatingManifold$ is at least $3$, we
  conclude that
  $\HomotopyGroupOfObject{1}{\partial V}{\Point}$ is isomorphic to the
  fundamental group of $K$.
  Fix a tubular neighborhood $\partial V\times
  I\subset \SeparatingManifold$ and denote the complement of $V \cup \partial V \times I$ by
  $V'$.
  Now $\NullBordism_i$ yields a relative bordism
  $
    \left(
      \NullBordism_i,
      V,
      V',
      \partial V \times I
    \right)
  $.

  Using the Seifert--Van Kampen Theorem one can conclude that the inclusion $V'
  \to \SeparatingManifold$ induces an isomorphism on fundamental groups and by
  Poincaré--Lefschetz Duality for twisted coefficients one concludes that the
  inclusion $V' \to \NullBordism_i$ is in fact a homotopy equivalence.
  Hence the aforementioned relative bordism is a relative s-cobordism  and
  hence by the relative s-cobordism theorem isomorphic
  to $V \times I$, furthermore the diffeomorphism between $\NullBordism_i$ and $V\times I$ can be chosen to be the identity on $V$ and hence induces the identity on
  $\pi_1$. Note that if we fix an orientation on $\SeparatingManifold$ and fix the
  induced orientation on
  $V$, then only one of the above diffeomorphisms can be orientation preserving
  and the other one is going to be orientation reversing.
  All in all this implies that $\Manifold$ is isomorphic to
  $\NullBordism_1\cup_{\Diffeomorphism} -\NullBordism_1$ and that this representation has
  the required properties.
  In particular the last point follows, because both inclusion $\SeparatingManifold \to
  \NullBordism_1$ and $\NullBordism_1 \to \Manifold$ induce an isomorphism on
  fundamental groups.
\end{proof}
Using this theorem we proceed by proving Theorem \ref{bigthm:A}.
\begin{proof}[Proof of Theorem~\ref{bigthm:A}]
  By Thom's solution of the rational Steenrod problem \cite{ThomQuelques}, there exists a $0\neq \Multiple\in
  \Integers$ such that $\Multiple \HomologyClass$ is representable by a closed
  orientable manifold $\Manifold$, i.e., there exists a map
  $
    \ContinuousMap'
    \colon
    \Manifold
    \to
    \ClassifyingSpace{\Group}
  $
  such that
  $
    \apply
      {\HomologyOfSpaceMorphism{\ContinuousMap'}}
      {\FundamentalClass{\Manifold}}
    =
    \Multiple \HomologyClass
  $.
  Since the dimension of $\Manifold$ is at least $5$, by performing surgery
  on $\Manifold$, we can assume that $\ContinuousMap'$ induces an isomorphism on
  fundamental groups.

  By Theorem~\ref{thm:TwistedDouble}, the manifold $\Manifold$ is diffeomorphic to a
  $\Diffeomorphism$-twisted double, where
  $
    \Diffeomorphism
    \colon
    \SeparatingManifold
    \to
    \SeparatingManifold
  $
  is a diffeomorphism and $\NullBordism$ a nullbordism of $\SeparatingManifold$,
  with the properties listed in that theorem.
  In particular by (\ref{itm:TwistedDoubleSIdentity}) the fundamental group of
  the mapping torus $\MappingTorus{\Diffeomorphism}$ is isomorphic to
  $\Group \times \Integers$. Let us consider the composition of the
  classifying map of the universal covering of
  $\MappingTorus{\Diffeomorphism}$ with the map induced by the projection
  $\Group \times \Integers\to \Group$ and denote it by
  $
    \ContinuousMap
    \colon
    \MappingTorus{\Diffeomorphism}
    \to
    \ClassifyingSpace{\Group}
  $.
  We will show that $\ContinuousMap'$ and $\ContinuousMap$ are bordant and hence they map
  the respective fundamental class to the same homology class in
  $\ClassifyingSpace{\Group}$.
  The mapping torus $\MappingTorus{\Diffeomorphism}$ is defined as
  $
    \SeparatingManifold
    \times
    [-1,1]
    /
    \sim
  $
  where we have $(-1,x)\sim (1,\apply{\Diffeomorphism}{x})$, let us denote
  $\SeparatingManifold \times [-0.5,0.5]$ by $N$.
  The manifold
  $
    \Bordism
    =
    \MappingTorus{\Diffeomorphism}
    \times
    [0,1]
    \cup_{N\times \{1\}}
    \NullBordism \times [-0.5,0.5]$
  gives a bordism between $\MappingTorus{\Diffeomorphism}$ and $\NullBordism
  \cup_\Diffeomorphism \NullBordism = \Manifold$.
  Furthermore an easy Seifert-Van-Kampfen argument shows that
  $\HomotopyGroupOfObject{1}{\Bordism}{\Point}$ is
  isomorphic to
  $
    \Group
    \times
    \Integers
  $.
  In particular, the composition of the classifying map for the universal
  covering of $\Bordism$ with the map corresponding to the projection
  $
    \Group \times \Integers \to \Integers
  $
  on fundamental groups yields a bordism between $\ContinuousMap'$ and
  $\ContinuousMap$.
\end{proof}
We conclude this section by deducing Corollary \ref{cor:mappingtori}.
\begin{proof}
  The map $\ContinuousMap$ in Theorem \ref{bigthm:A} induces an
  isometry in $\ellone$-homology because $\Integers$ is amenable. Hence the simplicial
  volume of the constructed mapping torus agrees with the $\ellone$-norm of the homology
  class $\HomologyClass\in \HomologyOfGroupObject{n}{\Group}{\Reals}$.

  Similarly, given a manifold $\Manifold$, the simplicial volume of $\Manifold$ agrees
  with the $\ellone$-norm of the image of the fundamental class in
  $
    \HomologyOfGroupObject
      {n}
      {\ClassifyingSpace{\HomotopyGroupOfObject{1}{\Manifold}{\ast}}}
      {\Reals}
  $%
  , because the classifying map of the universal covering induces the identity on
  fundamental groups. This concludes the first part of the corollary.

  Letting $\Manifold$ vary over all
  $2\HalfDimension+1$-manifolds, the main result of \cite{LoehHeuer}, which
  states that the spectrum of
  simplicial volume is dense in degrees at least $4$, implies the second assertion.
\end{proof}

\begin{rem}
Corollary \ref{cor:mappingtori} can be contrasted with recent work by L\"oh--Moraschini \cite[Cor.~4.10]{LoehMoraschini} that provides vanishing results for the simplicial volume of mapping tori with assumptions on the amenable category (a generalization of the LS-category) of the fiber.
 \end{rem}
 
 \begin{rem}
It should be noted that the manifolds $\Manifold$ in Corollary \ref{cor:mappingtori} can be assumed to be inessential, as the proof of Theorem \ref{bigthm:A}, upon which the proof of Corollary \ref{cor:mappingtori} is built, shows that these manifolds are nullbordant over $B\pi_1(\Manifold)$.
 \end{rem}

\section{Vanishing results}\label{Sec:NonEssentiality}
We next prove the two vanishing results outlined in the introduction, Theorems~\ref{thm:VanishingSimplicialVolume} and \ref{thm:NonEssentiality}, starting with the latter. For this, we make use of the following
corollary from \cite{WeinbergerFixedPoint} that we only state for connected manifolds and spell out in detail to make the notion of `rational isomorphism' completely explicit.
\begin{cor}[Rationally bijective self-maps of rationally essential manifolds]
\label{cor:Weinberger}
  Suppose $\Manifold$ is an oriented, rationally essential manifold, and the
  Novikov conjecture holds for $\HomotopyGroupOfObject{1}{\Manifold}{\Point}$
  and suppose that $\SmoothMap \colon \Manifold' \to \Manifold$, where
  $\Manifold'$ is also oriented, denotes a smooth map, satisfying:
  \begin{itemize}
  \item
    induces a map of positive sign on top homology (w.r.t.~the fixed orientations),
  \item
    induces an isomorphism on fundamental groups, and
  \item
    induces an isomorphism on all higher rational homotopy groups (or
    equivalently on all rational homology groups of the universal covering),
  \end{itemize}
  Then $\SmoothMap$ necessarily has degree $1$.
\end{cor}
\begin{proof}[Proof of Theorem~\ref{thm:NonEssentiality}]
We construct a self-map of $\ManifoldTotal$ that has degree bigger than $1$ and satsifies all three assumptions of Corollary~\ref{cor:Weinberger}. This implies that $\ManifoldTotal$ is rationally inessential.

  Let $\SmoothMap \colon \ManifoldBase \to \ManifoldBase$ denote a map of
  degree $d\neq -1, 0, 1$. By potentially replacing $\SmoothMap$ by $\SmoothMap
  \circ\SmoothMap$ we may assume that $d > 1$.
  We consider the following pullback diagram:
  \[
    \begin{tikzcd}
      \SmoothMap^* \ManifoldTotal
        \ar[r, "\SmoothMap_{\ManifoldTotal}"]
        \ar[d]
      &
      \ManifoldTotal
        \ar[d]
      \\
      \ManifoldBase
        \ar[r, "\SmoothMap"]
      &
      \ManifoldBase
    \end{tikzcd}
  \]
  We prove below that the map labeled $\SmoothMap_{\ManifoldTotal}$ has the
  three labeled properties from Corollary~\ref{cor:Weinberger} above. On the other hand it follows either from the Serre spectral sequence or by counting preimages of a regular value that this map has degree $d > 1$, so in conclusion it follows that $\ManifoldTotal$ cannot be
  rationally essential.

  The first condition of Corollary \ref{cor:Weinberger} holds, since the orientations of the fibers of the two bundles agree and $\SmoothMap$ has positive degree.

  Since $\ManifoldBase$ is $2$-connected, the long exact sequences in homotopy
  groups for both fiber bundles show that $\SmoothMap_{\ManifoldTotal}$ induces
  an isomorphism on fundamental groups as well.

  For the third assumption of Corollary \ref{cor:Weinberger}, consider the following diagram, obtained from the
  one above by applying the universal covering functor to it:
  \[
    \begin{tikzcd}
      \UniversalCovering{\SmoothMap^* \ManifoldTotal}
      \ar[r, "\UniversalCoveringMap{\SmoothMap_{\ManifoldTotal}}"]
      \ar[d]
      &
      \UniversalCovering{\ManifoldTotal}
      \ar[d]
      \\
      \ManifoldBase
      \ar[r, "\SmoothMap"]
      &
      \ManifoldBase
    \end{tikzcd}
  \]
  Since $\ManifoldBase$ is $2$-connected, we have that
  $\UniversalCovering{\SmoothMap^* \ManifoldTotal}$ is isomorphic to
  $\SmoothMap^* \UniversalCovering{\ManifoldTotal}$. Furthermore a map of
  non-zero degree has to be injective on rational cohomology, since the
  cup product pairing is non-singular.
  Hence $\SmoothMap$ induces an isomorphism on rational homology of
  $\ManifoldBase$. Combined this implies that the maps in the second diagram
  induce an isomorphism between the $E^2$ pages of the corresponding fiber
  bundles.
  This implies that $\UniversalCoveringMap{\SmoothMap_{\ManifoldTotal}}$
  induces an isomorphism on rational cohomology as required.
\end{proof}

\begin{rem}
  Consider the real projectivization of the tautological quaternionic bundle over $\Sphere{4} \cong \mathbf{H}P^1$. The pullback of the fiber sequence
  $
    SU(2)/ (\Integers / 2\Integers)
    \to
    \ClassifyingSpace{\Integers/2\Integers}
    \to
    \ClassifyingSpace{SU(2)}
  $
  along the inclusion $\Sphere{4} \to \ClassifyingSpace{SU(2)}$ of the
  $4$-skeleton yields the isomorphic smooth fiber bundle
  $
    \Reals P^3
    \to
    \Reals P^7
    \to
    \Sphere{4}
  $. The total space of this bundle is clearly essential.
  Thus, the `rational' in the conclusion of Theorem~\ref{thm:NonEssentiality}
  cannot be omitted.
\end{rem}

In order to prove Theorem~\ref{thm:VanishingSimplicialVolume} we will need the
notion of diameter of a simplex. If $\TopologicalSpace$ is a metric space and
$
  \sigma
  =
  \sum
  \sigma_i
$
is a chain in $\TopologicalSpace$, then we call
$
  \sup_i
  \apply{\mathrm{diam}}{\sigma_i}
$ the \introduce{diameter of $\sigma$}, where
$\apply{\mathrm{diam}}{\sigma_i}$ denotes the diameter of the image of
$\sigma_i$.
We will use the following lemma from \cite{YanoCircleAction}.
\begin{lem}[Lemma~1 in \cite{YanoCircleAction}] \label{lem:Yano}
  Let $\TopologicalSpace$ denote a compact metrized polyhedron, $m$ a
  non-negative integer and $\epsilon$ a positive number. Then there exist
  positive constant $\delta$ and $C$ such that:
  If a singular cycle $z\in C_m(\TopologicalSpace)$ is homologous to zero
  and its diameter is bounded by $\delta$, then there exists a chain $w\in
  C_{m+1}(\TopologicalSpace)$ such that $\partial w = z$, the diameter of
  $w$ is bounded by $\epsilon$ and $\SimplicialVolume{w} \leq C
  \SimplicialVolume{z}$.
\end{lem}
\begin{proof}[Proof of Theorem~\ref{thm:VanishingSimplicialVolume}]
  In a first step we will replace the bundle by one with the same
  simplicial volume which fibers over $\Sphere{2}$. Then we
  will leverage that the clutching function is defined over $\Sphere{1}$,
  which will allow us to construct very efficient cycles.

  Let $\ManifoldFiber \to \ManifoldTotal \to \Sphere{n}$ denote a bundle, then
  by trivializing the bundle on the upper and lower hemisphere one sees that
  the total space is a twisted double of $\Ball{n}\times \ManifoldFiber$
  twisted along a clutching function
  $
    \ClutchingFunction{\ManifoldTotal}
    \colon
    \Sphere{n-1}
    \to
    \DiffZero{\ManifoldFiber}
  $%
  . By abuse of notation, we also denote by $\ClutchingFunction{\ManifoldTotal}$
  the adjoint diffeomorphism from $\Sphere{n-1} \times \ManifoldFiber$ to
  itself.
  Completely analogous to the proof of Theorem~\ref{thm:TwistedDouble}, the
  manifold
  $
    \MappingTorus{\ClutchingFunction{\ManifoldTotal}}
    \times
    [0,1]
    \cup
    \Ball{n}
    \times
    \ManifoldFiber
    \times
    [0,1]
  $%
  , where $\Ball{n}\times \ManifoldFiber \times [0,1]$ is glued to the
  $[\frac{1}{3},\frac{2}{3}]$ part of the mapping torus,
  yields a bordism from $\ManifoldTotal$ to
  $\MappingTorus{\ClutchingFunction{\ManifoldTotal}}$ with fundamental group
  $\HomotopyGroupOfObject{1}{\ManifoldFiber}{\Point} \times \Integers$. Hence
  postcomposing the
  classifying map of its universal covering with the $\ellone$-norm preserving
  projection
  $
    \HomotopyGroupOfObject{1}{\ManifoldFiber}{\Point} \times \Integers
    \to
    \HomotopyGroupOfObject{1}{\ManifoldFiber}{\Point}
  $
  shows that $\MappingTorus{\ClutchingFunction{\ManifoldTotal}}$ has the same
  simplicial volume as $\ManifoldTotal$. Since
  $\ClutchingFunction{\ManifoldTotal}$ is the identity on the sphere component,
  the mapping torus fibers over $\Sphere{n-1}$. By repeating this process we
  get a bundle over $\Sphere{2}$, with the same simplicial volume as
  $\ManifoldTotal$.

  So let us now assume that we have a bundle
  $
    \ManifoldFiber
    \to
    \ManifoldTotal
    \to
    \Sphere{2}
  $%
  . In this case we have a clutching function
  $
    \Sphere{1}
    \to
    \DiffZero{\ManifoldFiber}
  $
  and a representative of the fundamental class of $\ManifoldTotal$ is given by
  a representative $\sigma$ of the fundamental class of $\Ball{2} \times
  \ManifoldFiber$ on the uper hemisphere and a chain that bounds
  $
    \apply
      {\ClutchingFunction{\ManifoldTotal}}
      {\partial \sigma}
  $
  on the lower hemisphere $\Ball{2}\times \ManifoldFiber$. We will construct
  representatives of $\ManifoldTotal$ whose $\ellone$-norm
  is arbitrarily small.

  Fix some metric on $\ManifoldFiber$ and let us denote the maximal Lipschitz
  constant of
  $
    \apply{\ClutchingFunction{\ManifoldTotal}}{-}
  $
  by $L$.
  By Lemma~\ref{lem:Yano} there exists some $\delta$ and a constant $C$
  such that any nullhomologous cycle $\sigma$ in $\Ball{2} \times
  \ManifoldFiber$ with diameter less than $\delta$ bounds a chain whose
  $\ellone$-norm is bounded by $C \SimplicialVolume{\sigma}$.
  Now pick some arbitrary representative $\sigma_{\ManifoldFiber}$ of the
  fundamental class of $\ManifoldFiber$ such that the diameter of all simplices
  is bounded by $\frac{\delta}{100*L}$, and some representative
  $\sigma_{\Sphere{1}}$ of the fundamental class of $\Sphere{1}$.
  Let
  $
    f^d
    \colon
    \Sphere{1} \times \ManifoldFiber
    \to
    \Sphere{1} \times \ManifoldFiber
  $
  denote the degree $d$-covering map given by $\apply{f^d}{(z,x)} = (z^d, x)$
  and let
  $
    \tau_d
    \colon
    C_*(\Sphere{1}\times \ManifoldFiber, \Reals)
    \to
    C_*(\Sphere{1}\times \ManifoldFiber,\Reals)
  $
  denote the corresponding transfer map, normalized such that
  $
    \HomologyOfSpaceMorphism{f_d}
    \circ
    \HomologyOfSpaceMorphism{\tau_d}
  $
  is the identity on homology.
  The transfer map preserves the $\ellone$ norm of a class while decreasing the
  diameter in $\Sphere{1}$-direction of all simplices. Because the topology of
  $\DiffZero{\ManifoldFiber}$ is generated by the metric of uniform
  convergence, one can choose a big enough $d$ such that
  $
    \apply
      {\HomologyOfSpaceMorphism{\tau_d}}
      {
        \apply
          {\HomologyOfSpaceMorphism{\ClutchingFunction{\ManifoldTotal}}}
          {\sigma_{\Sphere{1}} \times \sigma_{\ManifoldFiber}}
      }
  $
  has diameter less than $\delta$, hence it is bounded by a cycle $\rho$ such
  that
  \[
    \SimplicialVolume{\rho}
    \leq
    C \SimplicialVolume{
      \apply
        {\HomologyOfSpaceMorphism{\tau_d}}
        {
          \apply
            {\HomologyOfSpaceMorphism{\ClutchingFunction{\ManifoldTotal}}}
            {\sigma \times \sigma_{\ManifoldFiber}}
        }
    }
    =
    C \binom{m+1}{1}
    \SimplicialVolume{\sigma_{\Sphere{1}}}
    \SimplicialVolume{\sigma_{\ManifoldFiber}},
  \]%
  where $m$ denotes the dimension of the fiber.
  Now $\apply{\HomologyOfSpaceMorphism{f^d}}{\rho}$ bounds
  $
    \apply
      {\HomologyOfSpaceMorphism{\ClutchingFunction{\ManifoldTotal}}}
      {\sigma_{\Sphere{1}} \times \sigma_{\ManifoldFiber}}
  $
  and has norm bounded by
  $
    C \binom{m+1}{1}
    \SimplicialVolume{\sigma_{\Sphere{1}}}
    \SimplicialVolume{\sigma_{\ManifoldFiber}}
  $
  . By taking the cone of $\sigma_{\Sphere{1}}$ we
  end up with a representative of the fundamental class of $\Ball{2}$ denoted
  by $\sigma_{\Ball{2}}$.
  Hence
  \[
    \sigma_{\Ball{2}} \times \sigma_{\ManifoldFiber}
    +
    \apply{\HomologyOfSpaceMorphism{f^d}}{\rho}
  \]
  (Here the sum means that we take the first summand on the upper hemisphere
  and the second summand on the lower hemisphere)
  is a representative of the fundamental class of $E$ whose norm is bounded by
  \[
    C' \SimplicialVolume{\sigma_{\Sphere{1}}}
    \SimplicialVolume{\sigma_{\ManifoldFiber}}
  \]%
  Since the norm of $\sigma_{\Sphere{1}}$ can be chosen to be arbitrarily
  small this concludes the proof.
\end{proof}

\begin{rem}
\label{rem:boundedclassbackwardsimplication}
For hyperbolic groups, all cohomology classes are bounded \cite{Mineyev} and hence any rationally essential manifold has positive simplicial volume. Thus, in case $\Group$ is hyperbolic, Theorem \ref{thm:VanishingSimplicialVolume} implies that Theorem \ref{thm:NonEssentiality} also holds for base $S^2$.
\end{rem}
\section{Positive scalar curvature}\label{Sec:PSC}
In this section we prove Theorem~\ref{thm:PSC} by showing that the manifold
$\SeparatingManifold$ in the Theorem supports a psc-metric, whereas a certain
mapping torus of $\SeparatingManifold$ does not. We will prove the latter by
constructing a map of
non-zero degree from this mapping torus to the $7$-torus, following the line of
thought of the proof of Theorem \ref{thm:TwistedDouble}, and employing the
well-known machinery due to Schoen--Yau \cite{SchoenYau1, SchoenYau2}. The idea
how to conclude the desired statement appeared in recent work by Frenck
\cite{Frenck}.

\begin{proof}
We regard $\Torus{7} = (\Sphere{1})^7$ as a subset of $\ComplexNumbers^{7}$. The function
$
  \ContinuousMap\colon \Torus{7} \to \Reals,
  \left(z_1, \ldots, z_7\right)
  \mapsto
  \sum_i \RealPart{z_i}
  $
is a minimal Morse function. Its critical points have the coordinates $(\pm 1, \ldots, \pm 1)$, and the index of each of these is the number of $+1$'s. In particular, the critical values of $\ContinuousMap$ are the odd numbers between $-7$ and $7$.

Let us denote $\apply{\ContinuousMap^{-1}}{[-7, 0]}$ by $\NullBordism$
and $\apply{\ContinuousMap^{-1}}{[0,7]}$ by $\NullBordism'$.
Evidently $\NullBordism$ contains all critical points (and hence all handles)
of index $\leq 3$ and $\NullBordism'$ contains all critical points (and
all handles) of index $\geq 4$. Since $0$ is a regular value of $\ContinuousMap$, we conclude that $\SeparatingManifold = \apply{\ContinuousMap^{-1}}{\{0\}}$ is a smooth
submanifold with trivial normal bundle separating $\NullBordism$ and
$\NullBordism'$.

Let us investigate the geometric properties of $\SeparatingManifold$:
The manifolds $\SeparatingManifold$ and $\NullBordism$ inherit spin
structures from $\Torus{7}$, and hence $\SeparatingManifold$ is spin
null-bordant in $\Torus{7}$. Since
$
  \Torus{7}
  =
  \ClassifyingSpace{\Integers^7}
  =
  \ClassifyingSpace{\HomotopyGroupOfObject{1}{\SeparatingManifold}{\Point}}
$
this implies that $\SeparatingManifold$ carries a psc-metric by
\cite{GromovLawson}. (The theorem we need is formulated as Theorem 2.1 in
\cite{DwyerSchickStolz}.)
Since $\Sphere{2}\times \ComplexNumbers P^2$
is not spin, it follows that the manifold
$
  \Manifold
  =
  \SeparatingManifold
  \#
  (\Sphere{2} \times \ComplexNumbers P^2)
$
does not carry a spin structure as well. Nevertheless
$
  \Sphere{2}\times \ComplexNumbers P^2
$
carries a psc-metric; by Theorem~A in \cite{GromovLawson} the connected sum
carries a psc-metric as well.

In order to show that $\SpacePSC{\Manifold}$
has infinitely many components, we carry out the constructions in
the proof of Theorem~\ref{thm:TwistedDouble} by hand.

We need to construct a particular diffeomorphism between $\NullBordism$ and
$\NullBordism'$. In order to do that
let us denote
$
  \SeparatingManifold
  \cap
  \{
    (z_1,\ldots,z_7)\in \Torus{7}
    \mid
    \sum_i \ImaginaryPart{z_i} \leq 0
  \}
$
by $V$.
This is a manifold with boundary, denoted by $V_0$, consisting of the points
such that $\sum_i z_i = 0$.

Let
$
  \Psi \colon V \times \Sphere{1} \to \Torus{7}
$
denote the restriction of the diagonal action to $V$. Note that this map is
surjective, but it fails to be injective. Namely two points $(x,t)$ and
$(x',t')$ map to the same point in $\Torus{7}$ if and only if $x, x'\in V_0$
and $t'\cdot x' = t \cdot x$.
Let us denote by $\sim$ the equivalence relation generated by $(x,t) \sim
(t x, 1)$ for all $x \in V_0$. Then $\Psi$ descends to a homeomorphism
$\hat{\Psi}\colon V\times \Sphere{1}/ \hspace{-3pt} \sim \to \Torus{7}$ and by
construction $\NullBordism \cong V \times \Sphere{1}_{-}/ \hspace{-3pt} \sim$
and
$\NullBordism' \cong V \times \Sphere{1}_{+}/\hspace{-3pt} \sim$, where
$\Sphere{1}_{\pm}$ denotes the points in $\Sphere{1}$ with non-negative and
resp. non-positive imaginary part.

Fix a tubular neighborhood $V_0 \times [0,1] \subset V$ and let us define a
map
\begin{align*}
  \bar{\Diffeomorphism}
  \colon
  V \times \Sphere{1}_{-}/ \hspace{-3pt} \sim
  &
  \to
  V \times \Sphere{1}_{+}/ \hspace{-3pt} \sim
  \\
  [x,t]
  &
  \mapsto
  \begin{cases}
    [x,\overline{t}]
    &
    x \in (V_0 \times [0,1])^c
    \\
    [e^{2 \arg(t)(1-s)}x, s, \overline{t}]
    &
    [x,s,t] \in V_0 \times [0,1] \times \Sphere{1}_{-}/ \hspace{-3pt}\sim
  \end{cases}
\end{align*}
It is easy to verify that this yields a well-defined diffeomorphism.
In conclusion, this yields that the torus is diffeomorphic to a twisted double of
$V \times \Sphere{1}_{-}/ \hspace{-3pt}\sim$ and the gluing is given by the
restriction of $\bar{\Diffeomorphism}$ to $\SeparatingManifold \cong V\times \{
-1,
1\} / \hspace{-3pt} \sim$; we denote this restriction by $\Diffeomorphism$.
Since $\Diffeomorphism$ is the identity on $(V_0
\times [0,1])^c$, $\Diffeomorphism$ induces the identity on fundamental groups.

We will construct a particular diffeomorphism $\Diffeomorphism$ of $\SeparatingManifold$ that is the identity on some open set and such that the mapping
torus $\MappingTorus{\Diffeomorphism}$ has a map $\ContinuousMapALT$ of non-zero degree to $\Torus{7}$.

By the bordism construction in the proof of
Theorem~\ref{bigthm:A}, we conclude that the mapping
torus $\MappingTorus{\Diffeomorphism}$ has a map $\ContinuousMapALT$ to $B\Integers^7$
representing a non-zero homology class: in
other words, $\ContinuousMapALT$ is a map of non-zero degree to $\Torus{7}$. The restriction of $\Diffeomorphism$ to $\SeparatingManifold$ is the identity on an open subset, hence $\Diffeomorphism$ can be extended to a diffeomorphism $\hat{\Diffeomorphism}$ of $\Manifold \coloneqq \SeparatingManifold \# (\ComplexNumbers P^2 \times \Sphere{2})$.
Since $\Diffeomorphism$ is the identity on a neighborhood of the gluing of the
connected sum, there exists a map
$
  \ContinuousMapALT'
  \colon
  \MappingTorus{\hat{\Diffeomorphism}}
  \to
  \MappingTorus{\Diffeomorphism}
$ of degree $1$, and hence a map $\ContinuousMapALT \circ \ContinuousMapALT' \colon
\MappingTorus{\hat{\Diffeomorphism}}\to \Torus{7}$ of non-zero degree.
Using the Schoen--Yau minimal surface method (see p. 41f in
\cite{GromovLectures}), we conclude that
$\MappingTorus{\hat{\Diffeomorphism}}$ supports no psc-metric. Furthermore,
by the same argument no finite covering of
$\MappingTorus{\hat{\Diffeomorphism}}$ carries a psc-metric, which implies that
the subgroup of $\Diff{\Manifold}$ generated by $\hat{\Diffeomorphism}$ acts
freely on the path-components of $\SpacePSC{\Manifold}$: indeed, if for some $k
\geq 1$ the pull-back of the metric on $\Manifold$ along
$\hat{\Diffeomorphism}^k$ was isotopic to the original metric, then it would
follow that the $k$-sheeted covering $\MappingTorus{\hat{\Diffeomorphism}^k}$
of $\MappingTorus{\hat{\Diffeomorphism}}$ admits a psc-metric, which we ruled
out above. We thus deduce that
$
  \apply{\pi_0}{\SpacePSC{\Manifold}}
$
is infinite. The proof for $\SeparatingManifold$ instead of $\Manifold$ is completely analogous.
\end{proof}

Note that the same line of argument would have worked with any other
$\SeparatingManifold \subset \Torus{7}$ that is compatible with the
constructions of Theorem~\ref{thm:TwistedDouble}.

\section{Macroscopic dimension}\label{Sec:Essentiality}
We have seen above that rational group homology classes in odd
degrees at least $7$ can be represented by mapping tori. In this section, we
combine this insight with a result by Dranishnikov to construct examples of two closed manifolds with isomorphic fundamental groups whose universal coverings are diffeomorphic, but have different macroscopic dimensions. We begin by recalling this important concept.

\begin{dfn}[Uniformly cobounded]
A continuous map $f \colon X \to Y$ between metric spaces is called uniformly cobounded if there exists $s > 0$ such that the diameter of any pre-image $f^{-1}(y), y \in Y$, is bounded by $s$ from above.
\end{dfn}

\begin{dfn}[Macroscopic dimension]
\label{definition:macroscopic_dimension}
For a Riemannian manifold $(M,g)$, the macroscopic dimension of the universal cover, denoted $\MacroscopicDimension{\widetilde{\Manifold},g}$, is the smallest $k \geq 0$ such that there exists a uniformly cobounded Lipschitz map from $\widetilde{\Manifold}$ to a $k$-dimensional simplicial complex.
\end{dfn}

\begin{rem} (a) In fact,  $\MacroscopicDimension{\widetilde{\Manifold},g}$ is independent of the Riemannian metric $g$ on $\Manifold$, so we will simply write $\MacroscopicDimension{\widetilde{\Manifold}}$ henceforth. 

(b) One could also try to define the macroscopic dimension of any metric space $X$, but to get a working theory it then seems vital to assume that $X$ is a uniform simplicial complex, and require that the map $f$ from \ref{definition:macroscopic_dimension} is simplicial after iterated barycentric subdivisions, cf.~\cite[Def.~1.1]{Dranishnikov1}. One way to ensure this holds is restricting to smooth manifolds and requiring the maps to be Lipschitz. For elaborations on several related but somewhat inequivalent notions of macroscopic dimension, see \cite{Dranishnikov-clarification}.
\end{rem}

\begin{ex} If $\Manifold$ is the $n$-torus $(S^1)^n$, $\MacroscopicDimension{\widetilde{\Manifold}}$ is equal to $n$.
\end{ex}

\begin{prop}
\label{pro:inessential}
An inessential $\Dimension$-manifold $\Manifold$ satisfies $\mc{\Manifold} < \Dimension$.
\end{prop}
\begin{proof} Let $\pi_1(\Manifold)$ be denoted by $\Group$. Obstruction theory
allows us to deform the map $\Manifold \to B\Group$ classifying the universal
covering of $\Manifold$ along the inclusion of the $(\Dimension-1)$-skeleton
$B\Group^{(\Dimension -1)}$, see \cite[Prop.~3.2]{BolotovDranishnikov} for more details.

Then for any choice of a Riemannian metric on
$\Manifold$, the lift of this map to universal coverings,
$\widetilde{\Manifold} \to E\Group^{(\Dimension -1)}$, is uniformly cobounded,
hence $\mc{\Manifold} \leq \Dimension -1$.
\end{proof}

\begin{rem} Note that the conclusion of the previous proposition fails under the weaker assumption that $\Manifold$ is just rationally inessential \cite{Markinkowski}.

The converse of this statement, i.e., the assertion that any rationally essential manifold has maximal macroscopic dimension, was conjectured by Gromov \cite{Gromov_conjecture} and proven by Dranishnikov for amenable fundamental groups \cite{Dranishnikov1}, but shortly thereafter disproven in general \cite{Dranishnikov_counterexample}.
\end{rem}

We are now in a position to state and prove the main Theorem of this section.

 \begin{thm}[Twin manifolds with different macroscopic dimensions]
\label{thm:MacroscopicDimensionAmenable} In every dimension $2n+1 \geq 7$ and for any finitely presented amenable group $\Group$ such that $\HomologyOfSpaceObject{2n+1}{\Group}{\Rationals} \neq 0$,  there exist two closed $(2n+1)$-manifolds $\Manifold_1$ and $\Manifold_2$ such that (for any choice of Riemannian metrics):
\begin{enumerate}[(i)]
\item The fundamental groups $\pi_1(\Manifold_1)$ and $\pi_1(\Manifold_2)$ are isomorphic to $\Group \times \Integers$.
\item The universal coverings of $\widetilde{\Manifold_1}$ of $\Manifold_1$ and $\widetilde{\Manifold_2}$ of $\Manifold_2$ are diffeomorphic and quasi-isometric. 
\item $\mc{\Manifold_1} < 2n+1 = \mc{\Manifold_2}$.
\end{enumerate}
\end{thm}

\begin{proof} We start by applying Theorem \ref{bigthm:A} to $\Group$: let $\Manifold_1 \coloneqq \SeparatingManifold \times S^1, \Manifold_2 \coloneqq T_{\Diffeomorphism}$.
Then (i) is clear, (ii) holds as both universal coverings are diffeomorphic and quasi-isometric to $\widetilde{\SeparatingManifold} \times \Reals$.

From the construction in the proof of Theorem \ref{bigthm:A} it is clear that
the map $ \SeparatingManifold \to B\Group$ classifying the universal covering
of $ \SeparatingManifold$ is nullbordant, hence $\SeparatingManifold$ is
inessential. By Proposition \ref{pro:inessential}, this implies that
$\mc{\SeparatingManifold} < 2n$ and hence $\mc{\Manifold_1} < 2n+1$, as we can
cross any cobounded map $\widetilde{\SeparatingManifold} \to K^{(2n-1)}$ with
the identity on $\Reals$. The mapping torus $\Manifold_2  =
T_{\Diffeomorphism}$ however is rationally essential by construction. Recall
Gromov's conjecture stating that the universal cover of any rationally
essential manifold is of maximal macroscopic dimension which was verified for
amenable fundamental groups by Dranishnikov \cite{Dranishnikov1}. Applied to
the situation at hand we deduce that $\Manifold_2$ satisfies $\mc{
\Manifold_2}  = 2n+1$.
\end{proof}


\begin{rem}
A non-abelian example of $\Group$ as in Theorem \ref{thm:MacroscopicDimensionAmenable} is the integral Heisenberg group of rank $4n+1$, which is finitely presented and amenable. Its cohomology, calculated in \cite{HeisenbergGroup_Cohomology}, satisfies  $\HomologyOfSpaceObject{2n+1}{G}{\Rationals} \neq 0$ and $\HomologyOfSpaceObject{j}{\Group}{\Rationals} = 0$ for all $0 < j \leq 2n$.
\end{rem}
\section{Characteristic classes}\label{Sec:CharacteristicClasses}
In this final section we prove Theorem \ref{thm:CharClassesHurewicz}. We begin
by recalling that for an oriented smooth manifold bundle $\ManifoldFiber \to
\ManifoldTotal \to \ManifoldBase$, the \emph{vertical tangent bundle} vector
bundle $T_{\pi}$ over the total space $\ManifoldTotal$ whose dimension
$\Dimension$ agrees with the dimension of the fiber. For any characteristic
class of oriented vector bundles, $c \in
\CohomologyOfSpaceObject{\ast}{BSO(\Dimension)}{\Rationals}$, the class obtained
from fiber integration \[\kappa_{c} \coloneqq \int_{\pi} c(T_{\pi}) \in
\CohomologyOfSpaceObject{\ast-d}{\ManifoldBase}{\Rationals}\] is known as the
tautological class or generalized Milller--Morita--Mumford or simply
$\kappa$-class of $c$. These classes are characteristic classes of smooth fiber
bundles, so they
naturally live in the cohomology of $B\Diff{\ManifoldFiber}$.

Galatius--Randal-Williams generalized the concept of $\kappa$-classes to allow
for tangential structures and proved a higher-dimensional analogue of the
Madsen--Weiss theorem \cite{GRW}. Their theorems show that
the cohomology rings of diffeomorphism groups are, up to stabilization, generated by certain
cohomology classes defined on the level of MT-spectra; the classes $\kappa_{\alpha}$ can be understood in this context.

\begin{proof} For (1), we invoke the work by Galatius--Randal-Williams. We first construct a manifold with boundary that is weakly equivalent to  $B\Group^{(\HalfDimension)}$ and has trivial tangent bundle. From the finiteness assumption, we may assume that $B\Group^{(\HalfDimension)}$ is a finite CW-complex, hence it admits a local embedding into $\Reals^{2\HalfDimension}$. More precisely, we may assume that this local embedding has only finitely many self-intersections, all contained in top cells. We then consider a `tubular' open neighborhood of the image of $B\Group^{(\HalfDimension - 1)}$, and glue on a tubular neighborhood $C \times D^{\HalfDimension}$ for each ${\HalfDimension}$-cell $C$ of $B\Group^{(\HalfDimension)}$. This yields a manifold $\NullBordism$ with non-empty boundary  as desired. The $\HalfDimension$-stage Moore--Postnikov filtration of the classifying map of its tangent bundle relative to $\partial \NullBordism$ is $\NullBordism \to EO(2\HalfDimension) \times B\Group \xrightarrow{\theta} BO(2\HalfDimension)$. The relevant Thom spectrum for this tangental structure is $\textbf{MT}\theta = \sum^{\infty - 2\HalfDimension}_+ B\Group$. From the most general result in \cite{GRW} allowing for non-trivial fundamental groups (Theorem 1.8), we deduce the existence of some cobordism $\ManifoldAuxiliary$ starting in $\partial \NullBordism$ such that the class $\kappa_{\alpha}$ lies in the image of the map induced on cohomology from the classifying map $B\text{Diff}_{\partial}(\NullBordism \cup K) \to \Omega^{\infty} \textbf{MT}\theta = \Omega^{\infty} \sum^{\infty - 2\HalfDimension}_+ B\Group$ and is non-zero. (The cobordism $K$ is not constructed explicitly, but has to exist as the homotopy colimit from \cite[Theorem 1.8]{GRW} behaves in the way stated.) So we have shown that $0 \neq \kappa_{\alpha} \in \CohomologyOfSpaceObject{j}{ B\text{Diff}_{\partial}(\NullBordism \cup K)}{\Rationals}$.

To conclude, let $\Manifold$ be the double of $\NullBordism \cup K$. Then there is a map $B\text{Diff}_{\partial}(\NullBordism \cup K) \to B\DiffGroup{\Manifold}$ so that the class $\kappa_{\alpha}$ as in Theorem \ref{thm:CharClassesHurewicz} on the cohomology of the second space pulls back to the non-zero class on the left considered before, which finishes the proof.

Assertion (2) can be deduced from Theorem \ref{thm:NonEssentiality}: we need to prove that for any $\ManifoldFiber$-bundle $\pi \colon \ManifoldTotal \to \Sphere{\Dimension}$ with $\Dimension \ge 3$ and $\alpha \in \CohomologyOfSpaceObject{\Dimension + j}{B\Gamma}{\Rationals}$, we have $\kappa_{\alpha}(\pi) = 0$. Since
\[ \langle \kappa_{\alpha}(\pi) , [\Sphere{\Dimension}]\rangle = \langle \alpha, [E]\rangle = 0
\]
as $\ManifoldTotal$ is rationally inessential, this is indeed the case.
\end{proof}

\begin{rem}
Note that by applying Theorem \ref{thm:VanishingSimplicialVolume}, we see that the second part of Theorem \ref{thm:CharClassesHurewicz} also holds for bounded classes $\alpha \in \CohomologyOfSpaceObject{\Dimension + 2}{B\Gamma}{\Rationals}$. In particular, it holds for all homology classes in degree $d+2$ if $\Gamma$ is hyperbolic (compare Remark \ref{rem:boundedclassbackwardsimplication}).
\end{rem}

\bigskip
\printbibliography

\bigskip
{\footnotesize
\noindent
Thorben Kastenholz was supported by the DFG (German Research Foundation) -- SPP 2026, Geometry at Infinity - Project 73: Geometric Chern characters in $p$-adic equivariant $K$-theory.
\vspace{11pt}

\noindent
Jens Reinhold was supported by the DFG (German Research Foundation) -- SFB 1442 427320536, Geometry: Deformations and Rigidity, as well as under Germany's Excellence Strategy EXC 2044 390685587, Mathematics M\"unster: Dynamics-Geometry-Structure.
\newline
\noindent
\end{document}